\documentclass[a4paper]{article}

\usepackage[english]{babel}
\usepackage[utf8x]{inputenc}
\usepackage[T1]{fontenc}
\usepackage{physics}

\usepackage[a4paper,top=3cm,bottom=2cm,left=3cm,right=3cm,marginparwidth=1.75cm]{geometry}

\usepackage{tikz}
\usepackage{amsmath}
\usepackage{amssymb}
\usepackage{amsthm}
\usepackage{graphicx}
\usepackage{subcaption}
\usepackage{comment}
\usepackage[numbers]{natbib}
\usepackage[colorinlistoftodos]{todonotes}
\usepackage[colorlinks=true, allcolors=blue]{hyperref}
\newtheorem{theorem}{Theorem}[section]
\newtheorem{corollary}[theorem]{Corollary}

\newtheorem{proposition}[theorem]{Proposition}

\newtheorem{remark}[theorem]{Remark}
\newtheorem*{theorem*}{Theorem}
\newtheorem{definition}[theorem]{Definition}

\usepackage[affil-it,noblocks]{authblk}

\title{$C^*$-extreme maps and $*$-homomorphisms from $C(X)$ to finite von Neumann algebras}
\author[1]{Dimitrios Giannakis}
\author[2,$\ast$]{Michael Montgomery}
\author[3]{Travis Russell}

\affil[1]{Dartmouth College}
\affil[2]{Cascade Quantum}
\affil[3]{Texas Christian University}

\affil[$\ast$]{Corresponding author: mikmntgmry@gmail.com}
\date{}

\begin{document}
\maketitle

\begin{abstract}
    Given a unital inclusion of $C^*$-algebras $\mathcal C \subset \mathcal A$ and a unital inclusion $\mathcal C \subset \mathcal B$ into a von Neumann algebra $\mathcal B$, we investigate the extreme points of unital completely positive maps from $\mathcal A$ to $\mathcal B$ that fix $\mathcal C$ denoted by $UCP_\mathcal C(\mathcal A, \mathcal B)$. This space is obviously convex and $C^*$-convex with respect to the $C^*$-algebra $\mathcal C' \cap \mathcal B$. In this article we show that the $\mathcal C' \cap \mathcal B$-extreme points are exactly the $*$-homomorphisms that fix $\mathcal C$ when $\mathcal A$ is commutative and $\mathcal B$ has a normal faithful center valued trace. This generalizes a result due to Farenick and Morenz where $\mathcal C = \mathbb C 1$ and $\mathcal B = M_n(\mathbb C)$.
\end{abstract}

\section{Introduction}
The study of $C^*$-convexity began with Loebl and Paulsen in \cite{LOEBL198163} and Hopenwasser, Moore, and Paulsen in \cite{35443b83-20cf-3a55-b7e3-e96071c3959f}. Given a bimodule $\mathcal X$ over a $C^*$-algebra $\mathcal B$, we may consider the $C^*$-convex combinations $\sum_{i=1}^n b_i^* x_i b_i$ for elements $x_i \in \mathcal X$ and $b_i \in \mathcal B$ such that $\sum_{i=1}^n b_i^*b_i = 1_\mathcal B$. Further, a subset $K \subset \mathcal X$ is $\mathcal B$-convex if it is closed under $\mathcal B$-convex combinations. A combination is called proper if each coefficient $b_i$ is an invertible element of $\mathcal B$. An element $k \in K$ of an $\mathcal B$-convex set is called $\mathcal B$-extreme if the only proper $\mathcal B$-convex combinations yielding $k$ are of the form $k=u^*k'u$ for a unitary $u \in \mathcal B$.

The study of $C^*$-extreme points was advanced by Farenick and Morenz in \cite{1975eb9c-c8a0-3b9f-a03a-e542cfd075c6} and \cite{FM97} where they classify the $M_n(\mathbb C)$-extreme points of certain subsets of $n\times n$ matrices and in particular show that the $M_n(\mathbb C)$-extreme points of $UCP(C(X),M_n(\mathbb C))$ are exactly the $*$-homomorphisms. Gregg later generalized this characterization to the case $UCP(C(X),\mathcal K^+)$ where $\mathcal K^+$ is the algebra generated by the compact operators and scalar operators on a Hilbert space (see \cite{Gre09}).

The connection between $C^*$-extreme points and $*$-homomorphisms breaks down once we allow the target space to be all of $B(H)$ as shown by Farenick and Morenz in \cite{FM97}. In particular, they show that $\varphi(f) = p_H M_f|_H$ is $C^*$-extreme where $H=H^2(\mathbb T,m) \subset L^2(\mathbb T,m)$ is the Hardy space, $p_H$ the orthogonal projection to $H$ and $f \in C(\mathbb T)$. Clearly $\varphi$ is not a $*$-homomorphism.

More recently, Hotwani and Rao classified the von Neumann algebras $\mathcal M$ where the $C^*$-extreme points of the unit ball $\mathcal M_1$ consist of isometries and co-isometries (see \cite{HR26}). Such von Neumann algebras must be finite, properly infinite factors, or sums of the two.

The success of these generalizations, where the target algebra is $M_n(\mathbb C)$ or $\mathcal K^+$, suggests a new generalization to $*$-algebras with a trace. We require the target algebra, $\mathcal B$, be a von Neumann algebra to use von Neumann's bicommutant theorem. For traciality we focus on \textit{finite} von Neumann algebras (i.e. there exists a normal faithful center valued trace $E \colon \mathcal B \to Z(\mathcal B)$). Let $\mathcal A$ be a $C^*$-algebra which will be the domain of the UCP maps we consider. We will generalize several statements to spaces of UCP maps from $\mathcal A$ to $\mathcal B$ that fix $\mathcal C$ pointwise where $\mathcal C$ is a shared $C^*$-subalgebra $\mathcal C \subset \mathcal A$, $\mathcal C \subset \mathcal B$ with unital inclusions. Such UCP maps arise in many natural ways. For example, if $\mathcal A$ and $\mathcal B$ are represented on a shared Hilbert space $H$, the conditional expectations $\mathcal E \colon C^*(\mathcal A, \mathcal B) \to \mathcal B$ automatically fix $\mathcal C = \mathcal A \cap \mathcal B$.

 Consider the case where $\mathcal A$ is commutative. Then there are compact Hausdorff spaces $X$ and $Y$ with $\mathcal C =C(Y) \subset \mathcal A =C(X)$, $C(Y) \subset \mathcal B$. In this article we analyze the $C(Y)'\cap \mathcal B$-extreme points of
 $$UCP_{C(Y)}(C(X),\mathcal B) = \{ \varphi \in UCP(C(X),\mathcal B) : \varphi |_{C(Y)} = id_{C(Y)}\}.$$
 Our results are summarized in the following theorem.

\begin{theorem}\label{main theorem}
Let $X$ and $Y$ be compact Hausdorff spaces and $\mathcal B$ a finite von Neumann algebra. Then the $C(Y)' \cap \mathcal B$-extreme points of $UCP_{C(Y)}(C(X),\mathcal B)$ are exactly the $*$-homomorphisms from $C(X)$ to $\mathcal B$ that fix $C(Y)$.
\end{theorem}

In particular, the $\mathcal B$-extreme points of $UCP(C(X),\mathcal B)$ are exactly the $*$-homomorphisms from $C(X)$ to $\mathcal B$ when $\mathcal B$ is a finite von Neumann algebra and $X$ is a compact Hausdorff space.

\section{Bimodule Radon Nikodym type theorems}
For the remainder of this article $\mathcal C \subset \mathcal A$ is a unital inclusion of $C^*$-algebras, $\mathcal B$ is a von Neumann algebra faithfully and normally represented on $H$, and $\mathcal C \subset \mathcal B$ is a unital inclusion. Statements that require $\mathcal B$ to be a finite von Neumann algebra will be clearly indicated. Generic elements from $\mathcal A$, $\mathcal B$, $\mathcal C$, and the Hilbert space $H$ will be labeled $a_i$, $b_i$, $c_i$, and $h_i$ respectively.

For references on $C^*$-algebras and von Neumann algebras see \cite{Tak79}. Let $\lambda \colon \mathcal B \to B(H)$ be a faithful normal representation. Then $H$ can be equipped with a normal right action by $(\mathcal B')^{op}$, $\rho \colon \mathcal B' \to B(H)$. Then von Neumann's bicommutant theorem
$$\lambda(\mathcal B) = \{ T \in B(H) : [T,\rho(b')] = 0 \text{, for } b' \in \mathcal B'\} = \rho(\mathcal B')'$$
provides the main method to show that UCP maps send $\varphi \colon \mathcal A \to \lambda(\mathcal B)$.
Vector spaces with commuting left and right actions like $_\mathcal BH_{\mathcal B'^{op}}$ are called \textit{bimodules}. For the remainder of the article we will make use of bimodules and the relation between commutants above.

For standard background on complete positivity see Paulsen's book \cite{paulsen2002completely}. We briefly review some basic facts on complete positivity. A map $\varphi \colon \mathcal A \to \mathcal B$ is positive if $\varphi(\mathcal A_+) \subset \mathcal B_+$ where $\mathcal A_+$ and $\mathcal B_+$ denote the positive operators. A map $\varphi \colon \mathcal A \to \mathcal B$ is \textit{completely positive} (i.e. $\varphi \in CP(\mathcal A,\mathcal B)$) if $\varphi \otimes 1_{M_n(\mathbb C)} \colon \mathcal A \otimes M_{n}(\mathbb C) \to \mathcal B \otimes M_n(\mathbb C)$ is positive for all $n \in \mathbb N$. A map $\varphi$ is called a \textit{unital completely positive} map (UCP) if $\varphi(1_\mathcal A) = 1_\mathcal B$ along with being completely positive. We also make use of the Kadison-Schwarz inequality; if $\varphi$ is a UCP map then $\varphi(a^*a) \geq \varphi(a)^* \varphi(a)$ for all $a \in \mathcal A$. Finally, in \cite{Cho74}, Man-Duen Choi developed the multiplicative domain of $\varphi$ that is characterized by $\varphi(a^*a) = \varphi(a)^*\varphi(a)$ for $a \in \mathcal A$. If $a \in \mathcal A$ belongs to the multiplicative domain of $\varphi$ then $\varphi(xa) = \varphi(x)\varphi(a)$ for all $x \in \mathcal A$.

\begin{definition}
For $C^*$-algebras $\mathcal A$, $\mathcal B$, and $\mathcal C$ as above, $UCP_\mathcal C(\mathcal A,\mathcal B)$ denotes the set of $UCP$ maps $\varphi \colon \mathcal A \to \mathcal B$ that fix $\mathcal C$ pointwise.
\end{definition}

The space of operator norm bounded maps from $\mathcal A$ to $\mathcal B$, $\mathcal X = B(\mathcal A, \mathcal B)$, is a $\mathcal B -\mathcal B$ bimodule with left and right actions
$$(b_1 \cdot \phi \cdot b_2)(a)=b_1\phi(a)b_2$$
that contains $UCP_\mathcal C(\mathcal A, \mathcal B)$.

\begin{definition}
Let $_\mathcal B \mathcal X _\mathcal B$ be a topological vector space with a left and right action by a $C^*$-algebra $\mathcal B$. Fix a subalgebra $\mathcal C \subset \mathcal B$. A subset $\mathcal K \subset \mathcal X$ is called $\mathcal C' \cap \mathcal B$-convex if it is closed under the $\mathcal C' \cap \mathcal B$-convex combinations, i.e. $$\sum_{i=1}^n b_i^*k_ib_i \in \mathcal K$$ for all $k_i \in \mathcal K$, $b_i \in \mathcal C' \cap \mathcal B$, $1 \leq i \leq n$ with $\sum_{i=1}^n b_i^* b_i = 1$. When the $b_i$'s are invertible, the sum is called a proper $\mathcal C' \cap \mathcal B$-convex combination.
\end{definition}

The main example of $\mathcal C' \cap \mathcal B$-convex sets in this article is the collection of $UCP_\mathcal C(\mathcal A,\mathcal B)$ maps.

\begin{definition}
A point $k \in \mathcal K$ is called $\mathcal C' \cap \mathcal B$-extreme in $\mathcal K$ if, whenever $\sum_i b_i^* k_i b_i = k$ is a proper $\mathcal C' \cap \mathcal B$-convex combination, then each $k_i$ is unitarily equivalent to $k$ i.e. there are unitary operators $u_i \in \mathcal C' \cap \mathcal B$ such that $k_i = u_i^* k u_i$ for $1 \leq i \leq n$.
\end{definition}

 For the remainder of the paper $H$ will be a fixed normal faithful representation of $\mathcal B$ and $\mathcal B'^{op}$ is its commutant on $H$. We now show that the minimal Stinespring dilation of maps $\varphi \in UCP_\mathcal C(\mathcal A,\mathcal B)$ (i.e. $\varphi(\mathcal A) \subset \lambda(\mathcal B)$) inherits several bimodule properties.

\begin{proposition}
Let $\varphi \colon \mathcal A \to B(H)$ be a unital completely positive map and $\mathcal C \subset \mathcal A,\mathcal B$. Then $\varphi \in UCP_\mathcal C(\mathcal A,\mathcal B)$ (i.e. $\varphi(\mathcal A) \subset \lambda(\mathcal B)$ and $\varphi(c)= \lambda(c)$ for $c \in \mathcal C$) iff its minimal Stinespring dilation is given by a triple
$$\left(\pi \colon \mathcal A \to End(K_{\mathcal B'^{op}}), V \in Hom(_\mathcal C H_{\mathcal B'^{op}}, {_\mathcal C K_{\mathcal B'^{op}}}), {_\mathcal A K_{\mathcal B'^{op}}} \right), \quad \varphi(a) = V^* \pi(a) V, \quad a \in \mathcal A,$$
for a Hilbert space $K$ where the left action of $\mathcal A$ and $\mathcal C$ on $K$ are given by $\pi$ and $\mathcal B'^{op}$ has a normal action on the right.
\end{proposition}

\begin{proof}
Suppose the minimal Stinespring dilation of $\varphi$ is given by such a triple. Then $\varphi(a) = V^*\pi(a)V \in \rho(\mathcal B')' \cap B(H) = \lambda(\mathcal B)$. Since $V$ intertwines the actions of $\mathcal C$, then $\varphi(c) = V^*\pi(c)V = cV^*V = c 1_\mathcal B$ for all $c \in \mathcal C$. Hence, $\varphi \in UCP_\mathcal C(\mathcal A,\mathcal B)$.

Now suppose that $\varphi \in UCP_\mathcal C(\mathcal A,\mathcal B)$. Define $K_0 = \overline{\mathcal A \otimes H/\ker(\braket{\cdot}{\cdot})}$ with the inner product
$$\braket{\sum_i a'_i \otimes h'_i}{\sum_j a_j \otimes h_j} = \sum_{i,j} \braket{\varphi({a_j}^* a'_i) h'_i}{h_j}$$
$$\ker(\braket{\cdot}{\cdot}) = \{ \xi \in \mathcal A \otimes H | \braket{\xi}{\xi} = 0 \} = \{ \xi \in \mathcal A \otimes H | \braket{\xi}{\eta} = 0\text{ for all } \eta \in \mathcal A \otimes H \}.$$
Then we have maps $V_0 \colon H \to K_0$, $V_0(h) = [1 \otimes h]$ and $\pi_0 \colon \mathcal A \to B(K_0)$, $\pi_0(a)([a_1 \otimes h]) = [aa_1 \otimes h]$. The minimal dilation is the restriction to $K = [\pi(\mathcal A)VH] \subset K_0$ along with the restriction $\pi(a) = \pi(a)|_K$ and co-restriction $V = V|^K$. Then the triple $(\pi,V,K)$ is the minimal dilation of $\varphi$.

Since $\varphi$ maps to $\mathcal B \subset B(H)$, we may consider the right action by $\mathcal B'^{op}$ in this construction, $\rho(b')([a \otimes h]) = [a \otimes h b']$. It is trivial to verify that the minimal dilation consists of right $\mathcal B'^{op}$ module maps and the right $\mathcal B'^{op}$ module $K$, so that $V \colon H_{\mathcal B'^{op}} \to K_{\mathcal B'^{op}} $ and $\pi(a) \colon K_{\mathcal B'^{op}}  \to K_{\mathcal B'^{op}} $ for $a\in \mathcal A$. For normality, we need only check for norm bounded $SOT$ convergent nets $b_i' \to b'$ that $[a \otimes hb_i'] \to [a \otimes h b']$ which extends to $SOT$ convergence of $\rho(b_i') \to \rho(b')$. Furthermore, the fixed subalgebra $\mathcal C$ is given a left representation on $K$ by $\pi(c) \colon K_{\mathcal B'^{op}}  \to K_{\mathcal B'^{op}} $. Since $\mathcal C$ belongs to the multiplicative domain of $\varphi$, we also have $ac \otimes h- a \otimes ch \in \ker(\braket{\cdot}{\cdot})$ and so $V$ intertwines the left representation of $\mathcal C$, $\lambda|_\mathcal C \colon \mathcal C \to B(H)$ with $\pi|_\mathcal C$, $V \colon _\mathcal C H_{\mathcal B'^{op}}  \to {_\mathcal C K_{\mathcal B'^{op}} }$.
\end{proof}
We also note that $\mathcal C$ is faithfully represented on the minimal dilation. Since $V^*V=\varphi(1_\mathcal A) =1_\mathcal B$, we have
$$\lambda(c) = \lambda(c) V^*V  = V^* \pi(c) V.$$
If $\pi(c)=0$ for some $c \in \mathcal C$, then $\lambda(c) = V^*\pi(c)V = 0$ and so $c=0$ as $\lambda$ is a faithful representation of $\mathcal B$ on $H$. This makes $\pi \colon \mathcal C \to End(K_{\mathcal B'^{op}} )$ a faithful $*$-homomorphism.

An identical proposition holds for completely positive maps $\psi \colon \mathcal A \to B(H)$. Then $\psi \in CP(\mathcal A,\mathcal B)$ iff its minimal Stinespring dilation is given by a triple
$$\left(\alpha \colon \mathcal A \to End(L_{\mathcal B'^{op}}), W \in Hom(H_{\mathcal B'^{op}}, { L_{\mathcal B'^{op}}}), {_\mathcal A L_{\mathcal B'^{op}}} \right), \quad \psi(a) = W^* \alpha(a) W, \quad a \in \mathcal A.$$

\begin{remark}
Suppose we have two different minimal Stinespring dilations of $\varphi \colon \mathcal A \to B(H)$, $(\pi,V,K)$ and $(\pi',V',K')$. If $\varphi \in UCP_\mathcal C(\mathcal A,\mathcal B)$ then the unique unitary intertwining them is determined by
$$U(\pi(a)Vh) = \pi'(a)V'h$$
which means $U \in Hom({_\mathcal A K_{\mathcal B'^{op}}}, {_\mathcal A K'_{\mathcal B'^{op}}})$. We also have the identity $UV = V'$.
\end{remark}

We may now adapt many standard results about $C^*$-extreme points of $UCP$ maps to the context of $UCP_\mathcal C$. The following theorems and propositions are generalizations of known results from various papers whose proofs can be updated to the bimodule context. These are all within the context of unital $C^*$-algebras $\mathcal A$ and $\mathcal C$ along with a finite von Neumann algebra $\mathcal B$ as above. Their original statements correspond to the specialization $\mathcal B = B(H)$, $\mathcal B'^{op} =\mathbb C 1$, and $\mathcal C = \mathbb C 1$.

\begin{theorem}[Generalizing from \cite{Arveson69} A Radon Nikodym type Theorem]\label{Radon Nikodym type Theorem} Let $\varphi \colon \mathcal A \to B(H)$ with $\varphi \in UCP_{\mathcal C}(\mathcal A, \mathcal B)$ and minimal Stinespring dilation $(\pi,V, {_{\mathcal A}K_{\mathcal B'^{op}}})$. Then for $\psi \colon \mathcal A \to B(H)$, $\psi \in CP(\mathcal A,\mathcal B)$, $\psi \leq \varphi$ (i.e. $\varphi - \psi$ is completely positive) if and only if there exists a positive contraction $D \in \pi(\mathcal A)' \cap \rho(\mathcal B')' \cap B(K)$ such that $\psi(a) = V^*\pi(a)DV = V^*D^{1/2}\pi(a)D^{1/2}V$ for all $a \in \mathcal A$. The operator $D$ is unique in the sense that if $\psi(a) = V^*\pi(a)YV$ for some $Y \in \pi(\mathcal A)' \cap \rho(\mathcal B')'$, then $Y = D$.
\end{theorem}
\begin{proof}
The 'if' direction is immediate.

For the converse, let $(\alpha, W, {_{\mathcal A}L_{\mathcal B'^{op}}})$ be the minimal Stinespring dilation of $\psi$. Note that $W$ need not be an $\mathcal C$-module map. Define $T \colon _{\mathcal A}K_{\mathcal B'^{op}} \to {_{\mathcal A}L_{\mathcal B'^{op}}}$ by $T(\pi(a)Vh) = \alpha(a)Wh$ and extend linearly. $T$ is a well defined contraction since for every $\xi = \sum_{i=1}^n \pi(a_i)Vh_i$
$$\lVert T \xi \rVert^2 = \sum_{i,j=1}^n \braket{W^*\alpha(a_j^*a_i)Wh_i}{h_j} = \sum_{i,j=1}^n \braket{\psi(a_j^*a_i)h_i}{h_j}$$
$$\leq \sum_{i,j=1}^n \braket{\varphi(a_j^*a_i)h_i}{h_j} = \sum_{i,j=1}^n \braket{V^*\pi(a_j^*a_i)Vh_i}{h_j} = \lVert \xi \rVert^2.$$
This implies $T$ extends to the closure. Furthermore, $T$ is a $\mathcal A$-$\mathcal B'^{op}$ bimodule map as
$$\alpha(a_1)T(\pi(a_2)Vh)b' = \alpha(a_1)\alpha(a_2)Whb' \text{ for } a_i \in \mathcal A, \, h \in H, \, b' \in \mathcal B'.$$
Setting $D = T^*T$, we have a positive contraction $D \in \pi(\mathcal A)' \cap \rho(\mathcal B')'\cap B(K)$ such that $\psi(a) = V^*D\pi(a)V$ for all $a \in \mathcal A$. The uniqueness of $D$ follows from minimality of $(\alpha, W,L)$.
\end{proof}

\begin{theorem}[Generalizing from \cite{Arveson69}]
Suppose $\varphi \colon \mathcal A \to B(H)$ such that $\varphi \in UCP_\mathcal C(\mathcal A,\mathcal B)$ has minimal Stinespring dilation $(\pi,V,{_\mathcal A K_{\mathcal B'^{op}}})$. Then a necessary and sufficient condition for $\varphi$ to be scalar extremal in $UCP_\mathcal C(\mathcal A,\mathcal B)$ is that the map $End({_\mathcal A K_{\mathcal B'^{op}}}) \ni D \mapsto V^*DV \in \mathcal C' \cap \mathcal B$ is injective.
\end{theorem}
\begin{proof}
Suppose that $\varphi$ is extreme. Let $V^*DV=0$ for some $D \in \pi(\mathcal A)' \cap \rho(\mathcal B')' \cap B(K)$. Without loss of generality $-1 \leq D \leq 1$. Then $\varphi = \frac{1}{2}(\varphi_+ +\varphi_-)$ for $\varphi_\pm = V^*(1\pm D)\pi V \in UCP_\mathcal C(\mathcal A,\mathcal B)$. Since $\varphi$ is extreme, $\varphi = \varphi_+$. Therefore $V^*D\pi(a) V = 0$ for all $a \in \mathcal A$. By minimality, $\pi(\mathcal A)VH \subset K$ is dense and so $\braket{D\pi(a_1)Vh}{\pi(a_2)Vh'} = \braket{V^*D\pi(a_2^*a_1)Vh}{h'} =0$ and $D=0$.

For the reverse direction, suppose that $D \mapsto V^*DV$ is injective. If $\varphi = \frac{1}{2}(\varphi_1 +\varphi_2)$ for $\varphi_i \in UCP_\mathcal C(\mathcal A,\mathcal B)$, by the Radon Nikodym type theorem \ref{Radon Nikodym type Theorem} there exist positive contractions $D_i \in \pi(\mathcal A)' \cap \rho(\mathcal B')' \cap B(K)$ such that $\frac{1}{2}\varphi_i = V^*D_i \pi V$. Since $\varphi_i(1) = 1$, we have $V^*(2D_i -1)V = 0$. By injectivity $2D_i = 1$ and so $\varphi_i = \varphi$, proving extremality.
\end{proof}

The following characterization of $\mathcal C' \cap \mathcal B$-extreme points of $UCP_{\mathcal C}(\mathcal A, \mathcal B)$ is a generalization of Theorem 3.1 in \cite{FZ98} (see Theorem 3.3 in \cite{BBK21} for the operator valued measure version) where $\mathcal B = B(H)$ acts on $H$, $\mathcal B' = \mathbb C 1$, and $\mathcal C = \mathbb C 1$.
\begin{theorem}[Generalizing from \cite{FZ98}]\label{Abstract Characterization of C extreme points}
Let $\varphi \colon \mathcal A \to B(H)$ such that $\varphi \in UCP_{\mathcal C}(\mathcal A,\mathcal B)$ with minimal Stinespring dilation $(\pi,V, {_\mathcal AK_{\mathcal B'^{op}}})$. Then the following are equivalent:
\begin{enumerate}
\item $\varphi$ is a $\mathcal C' \cap \mathcal B$-extreme point.
\item For every operator $D \in \pi(\mathcal A)'\cap \rho(\mathcal B')' \cap B(K)$ such that $V^*DV$ is invertible, there is a unitary operator $U \in \pi(\mathcal A)'\cap \rho(\mathcal B')' \cap B(K)$ and invertible operator $Z \in \mathcal C' \cap \mathcal B$ such that $UD^{1/2}V = VZ$.
\end{enumerate}
\end{theorem}
\begin{proof}
Suppose that $\varphi$ is a $\mathcal C' \cap \mathcal B$-extreme point. Let $D \in \pi(\mathcal A)' \cap \rho(\mathcal B')' \cap B(K)$ such that $V^*DV$ is invertible. Then there exists a real number $r \in (0,1)$ such that $V^*D_iV$ is invertible for $D_1 = rD$ and $D_2 = 1-rD$. Clearly $D_1 + D_2 = 1$ and so $\varphi = V^*D_1^{1/2}\pi D_2^{1/2}V+V^*D_2^{1/2}\pi D_1^{1/2}V$. For each $i$, let $T_i = (V^*D_iV)^{1/2}$ Then $T_1$ and $T_2$ are invertible and
$$T_1^* \pi(c) T_1 +  T_2^* \pi(c) T_2  = \pi(c) (T_1^*T_1+ T_2^*T_2) = \pi(c)$$
since $T_i \in End({_\mathcal CH_{\mathcal B'^{op}}}) =\mathcal C' \cap \mathcal B$. Therefore
$$\varphi(\cdot) = \sum_{i=1}^2 T_i^*(T_i^{*-1}V^*D_i^{1/2}\pi(\cdot) D_i^{1/2} V T_i^{-1})T_i$$
is a representation of $\varphi$ as a proper $\mathcal C' \cap \mathcal B$-convex combination of maps
$$\varphi_i = (D_i^{1/2}VT_i^{-1})^*\pi(D_i^{1/2}VT_i^{-1}) \in UCP_\mathcal C(\mathcal A,\mathcal B).$$

Because $\varphi$ is a $\mathcal C' \cap \mathcal B$-extreme point, there is a unitary $W \in \mathcal C' \cap \mathcal B$ such that
$$V^*\pi V = U^* \varphi_1 U = (D_1^{1/2}VT_1^{-1}W)^*\pi(D_1^{1/2}VT_1^{-1}W).$$
Since all minimal Stinespring dilations are unitarily equivalent there is a unitary $U \in \pi(\mathcal A)'\cap \rho(\mathcal B')' \cap B(K)$ such that $U(D_1^{1/2}VT_1^{-1}W) = V$ and $U\pi(a) =\pi(a)U$ for all $a \in \mathcal A$. Hence $UD^{1/2}_1V = VZ_1$, where $Z_1 =W^*T_1$ is invertible in $\mathcal C' \cap \mathcal B$. Thus $UD^{1/2}V = VZ$, where $Z= \frac{1}{\sqrt{r}} Z_1$. This proves $1) \implies 2)$.

Suppose that the second statement holds and $\varphi = \sum_{i=1}^n b_i^*\varphi_ib_i$ be a proper $\mathcal C' \cap \mathcal B$-convex combination. For every $i$ $b_i^*\varphi_i b_i \leq \varphi$ and so there is a unique positive contraction $D_i \in \pi(\mathcal A)'\cap \rho(\mathcal B')' \cap B(K)$ such that $b_i^*\varphi_ib_i = V^*D_i^{1/2}\pi D_i^{1/2}V$. Because $b_i^*b_i$ is invertible, $(V^*D_iV)^{-1}$ exists. By hypothesis there exists a unitary $U_i \in \pi(\mathcal A)'\cap \rho(\mathcal B')' \cap B(K)$ and invertible operator $Z_i \in \mathcal C' \cap \mathcal B$ such that $U_i D_i^{1/2}V = VZ_i$. Then
$$\varphi_i = (D_i^{1/2}Vb_i^{-1})^*\pi (D_i^{1/2}Vb_i^{-1}) = (D_i^{1/2}Vb_i^{-1})^*\pi(U_i^*U_iD_i^{1/2}Vb_i^{-1})$$
$$=(U_iD_i^{1/2}Vb_i^{-1})^*\pi(U_iD_i^{1/2}Vb_i^{-1}) =(VZ_ib_i^{-1})^*\pi(VZ_ib_i^{-1}) = (Z_ib_i^{-1})^*\varphi (Z_ib_i^{-1}).$$
Since $\varphi(1) = \varphi_i(1) =1$, the operator $Z_ib_i^{-1}$ is an invertible isometry and so unitary.
\end{proof}

An immediate corollary of Theorem \ref{Abstract Characterization of C extreme points} gives one direction of the main theorem \ref{main theorem} of this article.

\begin{corollary}\label{hom corollary}
Every $*$-homomorphism $\varphi \colon \mathcal A \to B(H)$ such that $\varphi \in UCP_\mathcal C(\mathcal A,\mathcal B)$ is a $\mathcal C' \cap \mathcal B$-extreme point in $UCP_\mathcal C(\mathcal A,\mathcal B)$.
\end{corollary}
\begin{proof}
The minimal dilation is simply $(\varphi,id,{_\mathcal A H_{\mathcal B'^{op}}})$ where the left $\mathcal A$ action is given by $\varphi$. Then for any $D \in \varphi(\mathcal A)' \cap \rho(\mathcal B')' \cap B(H)$ that is invertible, we can take $U=id$ and $Z=D^{1/2}$.
\end{proof}

\begin{corollary}[Generalizing Theorem 3.1.5 in \cite{Zhou98}]\label{Invertible Extreme point criterion}
Let $\varphi \colon \mathcal A \to B(H)$ such that $\varphi \in UCP_{\mathcal C}(\mathcal A,\mathcal B)$. Then $\varphi$ is a $\mathcal C' \cap \mathcal B$-extreme point if an only if for any $\psi \in CP(\mathcal A, \mathcal B)$ with $\psi \leq \varphi$ and $\psi(1)$ invertible there exists an invertible operator $Z \in \mathcal C' \cap \mathcal B$ such that $\psi=Z^* \varphi Z$.
\end{corollary}
\begin{proof}
Assume that $\varphi$ is $\mathcal C' \cap \mathcal B$-extreme. Let $\psi \in CP(\mathcal A, \mathcal B)$ with $\psi \leq \varphi$ and $\psi(1)$ invertible. Take the minimal Stinespring dilation of $\varphi$, $(\pi,V,{_\mathcal A K_{\mathcal B'^{op}}})$. By theorem \ref{Radon Nikodym type Theorem} there is a $D \in \pi(\mathcal A)' \cap \rho(\mathcal B')' \cap B(K)$ such that $\psi = V^*D\pi V$. Since $\psi(1) = V^*DV$ is invertible, we can apply theorem \ref{Abstract Characterization of C extreme points}. There exists a unitary $U \in \pi(\mathcal A)' \cap \rho(\mathcal B')' \cap B(K)$ and an invertible operator $Z \in \mathcal C' \cap \mathcal B$ such that $UD^{1/2}V=VZ$. Hence,
$$\psi(a) = V^*D^{1/2}\pi(a)D^{1/2}V = V^*D^{1/2}\pi(a)U^*UD^{1/2}V = V^*D^{1/2}U^*\pi(a)UD^{1/2}V$$
$$=(VZ)^*\pi(a)(VZ) = Z^*\varphi(a)Z.$$

For the converse, let $\varphi = \sum_{i=1}^n T_i^* \varphi_i T_i$ be a proper $\mathcal C' \cap \mathcal B$-convex combination. Then $T_i^*\varphi_iT_i \leq \varphi$ for all $i=1,...,n$ and each $T_i$ is invertible. Therefore $T_i^*\varphi_i(1)T_i$ is invertible. By hypothesis, there are invertible operators $Z_i \in \mathcal C' \cap \mathcal B$ such that
$$T_i^*\varphi_iT_i = Z_i^*\varphi Z_i.$$
Defining $U_i = Z_iT_i^{-1}$ we see that
$$\varphi_i = U_i^* \varphi U_i$$
and $U_i^*U_i = U_i^*\varphi(1)U_i = \varphi_i(1)=1$ making $U_i$ unitary. Hence $\varphi$ and $\varphi_i$ are unitarily equivalent.
\end{proof}

The following proposition appears in \cite{FPS11} where $\mathcal B = B(H)$ and $H$ is a finite dimensional Hilbert space. The same proof applies to $\mathcal B$ a finite von Neumann algebra due to the presence of traces.

\begin{proposition}[Generalizing Proposition 2.1 in \cite{FPS11}]
Let $\mathcal B$ be a finite von Neumann algebra. The $\mathcal C' \cap \mathcal B$-extreme points of $UCP_\mathcal C(\mathcal A,\mathcal B)$ are also scalar extreme points.
\end{proposition}
\begin{proof}
Suppose $\varphi \colon \mathcal A \to B(H)$ such that $\varphi \in UCP_\mathcal C(\mathcal A,\mathcal B)$ is a $\mathcal C' \cap \mathcal B$-extreme point and $\varphi = \lambda \varphi_1 +(1-\lambda)\varphi_2$ for some $\varphi_i \colon \mathcal A \to B(H)$ such that $\varphi_i \in UCP_\mathcal C(\mathcal A,\mathcal B)$ and $\lambda \in (0,1)$. Define $T_1 = \sqrt{\lambda} 1_\mathcal B , T_2 = \sqrt{1-\lambda} 1_\mathcal B \in \mathcal C' \cap \mathcal B$ making $\varphi$ a proper $\mathcal C' \cap \mathcal B$-convex combination $\varphi = \sum_{i=1}^2T_i^* \varphi_i T_i$. Hence, there are unitaries $U_i \in \mathcal C' \cap \mathcal B$ such that $\varphi_i = U_i^* \varphi U_i$. Then for any $a \in \mathcal A$, $\varphi(a) = \lambda U_1^* \varphi(a) U_1 + (1-\lambda) U_2^* \varphi(a) U_2$.

Since $\mathcal B$ is finite, it has a normal faithful center valued trace $E \colon \mathcal B \to Z(\mathcal B)$. For every normal state $\omega$ on $Z(\mathcal B)$ define $\tau_\omega = \omega \circ E$. Then $\varphi(a)$, $U_1^*\varphi(a)U_1$, and $U_2^*\varphi(a)U_2$ belong to the sphere of radius $\lVert \varphi(a)\rVert_{\tau_\omega,2}$ in $L^2(\mathcal B,\tau_\omega)$. Since the sphere of a Hilbert space contains no nontrivial triples from a line, these three terms coincide in $L^2(\mathcal B,\tau_\omega)$. Since
$$\lVert \varphi(a) - \varphi_1(a)\rVert_{\tau_\omega,2} = 0 = \lVert \varphi(a) - \varphi_2(a)\rVert_{\tau_\omega,2}$$
for every $a \in \mathcal A$, and every $\omega \in Z(\mathcal B)_*$, $\varphi = \varphi_1 = \varphi_2$ and $\varphi$ is an extreme point of $UCP_\mathcal C(\mathcal A,\mathcal B)$.
\end{proof}

\section{Proof of the main theorem}

We can now prove the other direction of the main theorem \ref{main theorem}.

\begin{proposition}
Let $X$ and $Y$ be compact Hausdorff spaces, $\varphi \in UCP_{C(Y)}(C(X), \mathcal B)$ a $C(Y)' \cap \mathcal B$-extreme point and $\mathcal B$ a finite von Neumann algebra with normal faithful center valued trace $E$. Then $\varphi$ is a $*$-homomorphism.
\end{proposition}

\begin{proof}
Let $g \in C(X)$ be a strictly positive element bounded away from $0$ and $1$, such that $\varepsilon  \leq g \leq 1-\varepsilon$ for some $\varepsilon > 0$. We define a map $\psi : C(X) \to \mathcal B$ by
\[ \psi(f) = \varphi(gf). \]
By commutativity and positivity of $g$, $\psi$ is a completely positive map. Furthermore, the map $\varphi - \psi$, given by $(\varphi - \psi)(f) = \varphi((1-g)f)$, is also completely positive because $1-g \geq \varepsilon  > 0$. Thus, $\psi \leq \varphi$. Finally, we note that $\psi(1) = \varphi(g)$ is an invertible, positive operator in $\mathcal B$  as it is bounded below by $\varepsilon $.

By Corollary \ref{Invertible Extreme point criterion}, there exists an invertible operator $Z \in C(Y)'  \cap \mathcal B$ such that
$$\varphi(gf) = Z^* \varphi(f) Z \quad \text{for all } f \in C(X)$$
and so for $f=1$, $\varphi(g) = Z^* Z$.
We can iterate this relation by evaluating the map at $g^2 f$:
$$ \varphi(g^2 f) = \varphi(g(gf)) = Z^* \varphi(gf) Z = Z^* (Z^* \varphi(f) Z) Z = (Z^*)^2 \varphi(f) Z^2. $$
Evaluating this at $f=1$, we obtain $\varphi(g^2) = (Z^*)^2 Z^2$.
Since $\varphi$ is a UCP map, the Kadison-Schwarz inequality implies that $\varphi(g^2) \geq \varphi(g)^2$ and so
$$(Z^*)^2 Z^2 \geq (Z^* Z)^2.$$

We continue by showing normality of $Z$. Take the polar decomposition $Z = uD$, where $D = (Z^* Z)^{1/2} = \varphi(g)^{1/2}$. Because $Z$ is invertible, $D$ is strictly positive and $u$ is unitary. Substituting $Z = uD$ into the previous inequality gives
$$D u^* D^2 uD \geq D^4. $$
Since $D$ is invertible, $u^*D^2u \geq D^2$ as
$$\braket{(u^*D^2u-D^2)\xi}{\xi} = \braket{(Du^*D^2uD-D^4)D^{-1}\xi}{D^{-1}\xi}\geq 0 \text{ for all } \xi\in H.$$

Since $u^*D^2u-D^2 \geq 0$ and $E(u^*D^2u-D^2)=0$, faithfulness of the center valued trace implies that $u^*D^2u=D^2$.

Because $D$ is the unique positive square root of $D^2$, $u$ must also commute with $D$. Since its unitary part commutes with its positive part, $Z$ is a normal operator.

Because $Z$ is normal, $Z^* Z = Z Z^*$. Returning to $\varphi(g)$ and $\varphi(g^2)$, we find
$$\varphi(g^2) = (Z^*)^2 Z^2 = Z^* Z^* Z Z = (Z^* Z)^2 = \varphi(g)^2.$$
Since all positive operators $g \in C(X)$ with $\varepsilon  \leq g \leq 1-\varepsilon$ belong the multiplicative domain of $\varphi$ and these positive elements span $C(X)$, the multiplicative domain of $\varphi$ is the entirety of $C(X)$. Thus, $\varphi$ is a $*$-homomorphism.
\end{proof}

Combining this result with Corollary $\ref{hom corollary}$ yields a classification of $C(Y)' \cap \mathcal B$-extreme points of $UCP_{C(Y)}(C(X),\mathcal B)$ when $\mathcal B$ is a finite von Neumann algebra. Hence, proving the main theorem \ref{main theorem}. This naturally leads us to the type $III$ question. Can we build a $\mathcal B$-extremal UCP map $\varphi \colon C(X) \to \mathcal B$ where $\mathcal B$ is a type $III$ von Neumann algebra and $\varphi$ is not a $*$-homomorphism?

\section*{Acknowledgments}

Dimitrios Giannakis acknowledges support from the U.S.\ Department of Defense, Basic Research Office under Vannevar Bush Faculty Fellowship grant N00014-21-1-2946 and the U.S.\ Department of Energy under grant DE-SC0025101.
Michael Montgomery and Travis Russell were supported as postdoctoral fellows from the first grant. Travis Russell was also supported by a grant from the United States-Israel Binational Science Foundation (BSF-2024161), Jerusalem, Israel, and by a grant from the TCU Research and Creative Activities Fund.

\bibliographystyle{plain}
\bibliography{references}

\end{document}